\documentclass[12pt]{amsart}

\usepackage{geometry} 
\usepackage[pdfpagemode=UseNone,pdfstartview=FitH]{hyperref}

\newcommand{\A}{{\mathcal A}}
\newcommand{\bdry}[1]{\partial #1}

\newcommand{\eps}{\varepsilon}

\newcommand{\incl}{\hookrightarrow}
\newcommand{\N}{\mathbb N}
\newcommand{\norm}[2][]{\left\|#2\right\|_{#1}}
\newcommand{\PS}[1]{$(\text{PS})_{#1}$}
\newcommand{\pnorm}[2][]{\if #1'' \left|#2\right|_p \else \left|#2\right|_{#1} \fi}
\newcommand{\R}{\mathbb R}
\newcommand{\RP}{\R \text{P}}

\newcommand{\set}[1]{\left\{#1\right\}}
\newcommand{\vol}[1]{\left|#1\right|}
\newcommand{\Z}{\mathbb Z}

\newenvironment{enumroman}{\begin{enumerate}
\renewcommand{\theenumi}{$(\roman{enumi})$}
\renewcommand{\labelenumi}{$(\roman{enumi})$}}{\end{enumerate}}

\newenvironment{properties}[1]{\begin{enumerate}
\renewcommand{\theenumi}{$(#1_\arabic{enumi})$}
\renewcommand{\labelenumi}{$(#1_\arabic{enumi})$}}{\end{enumerate}}

\newtheorem{corollary}{Corollary}[section]

\newtheorem{theorem}[corollary]{Theorem}

\theoremstyle{remark}

\numberwithin{equation}{section}

\title[Elliptic problems with critical exponential growth]{Multiplicity results for elliptic problems with critical exponential growth}

\author{Kanishka Perera}
\address{Kanishka Perera: Department of Mathematics, Florida Institute of Technology, 150 W University Blvd, Melbourne, FL 32901-6975, USA}
\email{kperera@fit.edu}

\date{}


\begin{document}

\begin{abstract}
We prove new multiplicity results for some elliptic problems with critical exponential growth. More specifically, we show that the problems considered here have arbitrarily many solutions for all sufficiently large values of a certain parameter $\mu > 0$. In particular, the number of solutions goes to infinity as $\mu \to \infty$. The proof is based on an abstract critical point theorem.
\end{abstract}

\maketitle

\section{Introduction}

Let $\Omega$ be a bounded domain in $\R^2$ and consider the problem
\begin{equation} \label{1}
\left\{\begin{aligned}
- \Delta u & = h(u)\, e^{\alpha_0 u^2} && \text{in } \Omega\\[10pt]
u & = 0 && \text{on } \bdry{\Omega},
\end{aligned}\right.
\end{equation}
where $\alpha_0 > 0$ and $h$ is a continuous function on $\R$ satisfying $h(0) = 0$ and
\begin{equation} \label{2}
\lim_{|t| \to \infty}\, \frac{h(t)}{e^{\alpha t^2}} = 0
\end{equation}
for all $\alpha > 0$. This problem is critical with respect to the Trudinger-Moser inequality
\[
\sup_{u \in H^1_0(\Omega),\, \norm{u} \le 1}\, \int_\Omega e^{4 \pi u^2}\, dx < \infty
\]
(see Trudinger \cite{MR0216286} and Moser \cite{MR0301504}), which makes finding nontrivial solutions challenging. Set $f(t) = h(t)\, e^{\alpha_0 t^2}$ and let $F(t) = \int_0^t f(s)\, ds$ be the primitive of $f$. Assume that
\begin{properties}{A}
\item \label{A1} $0 < 2F(t) \le t f(t)$ for all $t \in \R \setminus \set{0}$,
\item \label{A2} there exist $t_0, M > 0$ such that $F(t) \le M\, |f(t)|$ for $|t| \ge t_0$.
\end{properties}
Let $\lambda_1 > 0$ be the first Dirichlet eigenvalue of the Laplacian in $\Omega$ and let $d$ be the radius of the largest open ball contained in $\Omega$. The following theorem was proved in de Figueiredo et al.\! \cite{MR1386960,MR1399846} and improves an earlier result of Adimurthi \cite{MR1079983}.

\begin{theorem} \label{Theorem 1}
Assume \eqref{2}, \ref{A1}, and \ref{A2}. If
\begin{equation} \label{11}
\limsup_{t \to 0}\, \frac{2F(t)}{t^2} < \lambda_1
\end{equation}
and
\begin{equation} \label{16}
\liminf_{|t| \to \infty}\, t h(t) > \frac{2}{\alpha_0 d^2},
\end{equation}
then problem \eqref{1} has a nontrivial solution.
\end{theorem}

In particular, this theorem has the following corollary for the problem
\begin{equation} \label{3}
\left\{\begin{aligned}
- \Delta u & = \mu\, h(u)\, e^{\alpha_0 u^2} && \text{in } \Omega\\[10pt]
u & = 0 && \text{on } \bdry{\Omega},
\end{aligned}\right.
\end{equation}
where $\mu > 0$ is a parameter.

\begin{corollary} \label{Corollary}
Assume \eqref{2}, \ref{A1}, and \ref{A2}. If
\begin{equation} \label{5}
\lim_{t \to 0}\, \frac{2F(t)}{t^2} = 0
\end{equation}
and
\begin{equation} \label{4}
\liminf_{|t| \to \infty}\, t h(t) > 0,
\end{equation}
then problem \eqref{3} has a nontrivial solution for all sufficiently large $\mu > 0$.
\end{corollary}

Various extensions of these results and related results can be found in \cite{MR1093706,MR2772124,MR3130547,MR2286391,MR4263694,MR3145918,MR1392090} and in their references. However, two basic questions concerning problem \eqref{3} have remained largely open over the last three decades:
\begin{enumerate}
\item Can the assumption \eqref{4} be removed?
\item Are there multiple nontrivial solutions of problem \eqref{3} for large $\mu$ when $h$ is an odd function?
\end{enumerate}
In this note we give positive answers to both of these questions. We will show that, even without the assumption \eqref{4}, problem \eqref{3} has arbitrarily many solutions for all sufficiently large $\mu > 0$ when $h$ is odd. Our main result is the following theorem.

\begin{theorem} \label{Theorem 2}
Assume \eqref{2}, \ref{A1}, \ref{A2}, and $h(- t) = - h(t)$ for all $t \in \R$. If \eqref{5} holds, then given any $m \in \N$, there exists $\mu_m > 0$ such that problem \eqref{3} has $m$ distinct pairs of nontrivial solutions $\pm u_1,\dots,\pm u_m$ for all $\mu > \mu_m$. In particular, the number of solutions of problem \eqref{3} goes to infinity as $\mu \to \infty$.
\end{theorem}

A model problem is
\begin{equation} \label{6}
\left\{\begin{aligned}
- \Delta u & = \mu\, |u|^{r - 2}\, u\, e^{\alpha_0 u^2 + \beta |u|} && \text{in } \Omega\\[10pt]
u & = 0 && \text{on } \bdry{\Omega},
\end{aligned}\right.
\end{equation}
where $\alpha_0 > 0$, $r > 2$, $\beta \in \R$, and $\mu > 0$. The corresponding function $h(t) = |t|^{r - 2}\, t\, e^{\beta |t|}$ satisfies \eqref{2} and is odd. The assumptions \ref{A1} and \ref{A2} are easily verified if $\beta \ge 0$, or $\beta < 0$ and $\beta^2 \le 8 \alpha_0\, (r - 2)$. However,
\[
\lim_{|t| \to \infty}\, t h(t) = \begin{cases}
+ \infty & \text{if } \beta \ge 0\\[5pt]
0 & \text{if } \beta < 0,
\end{cases}
\]
so Theorem \ref{Theorem 1} gives a nontrivial solution for all $\mu > 0$ when $\beta \ge 0$, but neither Theorem \ref{Theorem 1} nor Corollary \ref{Corollary} gives a nontrivial solution for any $\mu > 0$ when $\beta < 0$. In contrast, Theorem \ref{Theorem 2} gives arbitrarily many solutions for all sufficiently large $\mu > 0$ when $\beta \ge 0$ and when $\beta < 0$ and $\beta^2 \le 8 \alpha_0\, (r - 2)$.

\begin{theorem}
If $r > 2$, then given any $m \in \N$, there exists $\mu_m > 0$ such that problem \eqref{6} has $m$ distinct pairs of nontrivial solutions $\pm u_1,\dots,\pm u_m$ for all $\mu > \mu_m$ in each of the following cases:
\begin{enumroman}
\item $\beta \ge 0$,
\item $\beta < 0$ and $\beta^2 \le 8 \alpha_0\, (r - 2)$.
\end{enumroman}
In particular, the number of solutions goes to infinity as $\mu \to \infty$ in both cases.
\end{theorem}

We also prove the following variant of Theorem \ref{Theorem 2} for problem \eqref{1}, which assumes \eqref{11} instead of \eqref{5} and therefore can be applied to the limiting case $r = 2$ of the model problem \eqref{6}. Let $\lambda_1 < \lambda_2 \le \lambda_3 \le \cdots$ be the sequence of Dirichlet eigenvalues of $- \Delta$ in $\Omega$, repeated according to multiplicity.

\begin{theorem} \label{Theorem 4}
Assume \eqref{2}, \ref{A1}, \ref{A2}, and $h(- t) = - h(t)$ for all $t \in \R$. If \eqref{11} holds and
\begin{equation} \label{12}
F(t) \ge \mu\, |t|^s \quad \forall t \in \R
\end{equation}
for some $s > 2$ and
\begin{equation} \label{13}
\mu > \frac{1}{s} \left[\left(\frac{1}{2} - \frac{1}{s}\right) \frac{\alpha_0}{2 \pi} \vol{\Omega}\right]^{s/2 - 1} \lambda_m^{s/2},
\end{equation}
then problem \eqref{1} has $m$ distinct pairs of nontrivial solutions $\pm u_1,\dots,\pm u_m$.
\end{theorem}

For example, consider the model problem
\begin{equation} \label{14}
\left\{\begin{aligned}
- \Delta u & = \lambda u\, e^{\alpha_0 u^2 + \beta |u|} && \text{in } \Omega\\[10pt]
u & = 0 && \text{on } \bdry{\Omega},
\end{aligned}\right.
\end{equation}
where $\alpha_0 > 0$, $0 < \lambda < \lambda_1$, and $\beta \ge 0$. Theorem \ref{Theorem 1} gives a nontrivial solution of this problem for all $\beta \ge 0$. In contrast, since
\[
\int_0^t \lambda s\, e^{\alpha_0 s^2 + \beta |s|}\, ds \ge \frac{\lambda \beta}{3}\, |t|^3 \quad \forall t \in \R,
\]
Theorem \ref{Theorem 4} gives the following multiplicity result for large $\beta > 0$.

\begin{theorem}
If $0 < \lambda < \lambda_1$ and
\[
\beta > \frac{1}{2 \lambda} \left(\frac{\alpha_0}{3 \pi} \vol{\Omega}\right)^{1/2} \lambda_m^{3/2},
\]
then problem \eqref{14} has $m$ distinct pairs of nontrivial solutions $\pm u_1,\dots,\pm u_m$. In particular, the number of solutions goes to infinity as $\beta \to \infty$.
\end{theorem}

As an example where the assumption \eqref{16} does not hold, consider the problem
\begin{equation} \label{15}
\left\{\begin{aligned}
- \Delta u & = (\lambda + \mu\, |u|)\, u\, e^{\alpha_0 u^2 - \beta |u|} && \text{in } \Omega\\[10pt]
u & = 0 && \text{on } \bdry{\Omega},
\end{aligned}\right.
\end{equation}
where $\alpha_0 > 0$, $0 < \lambda < \lambda_1$, $\beta > 0$, and $\mu > 0$. The corresponding function $h(t) = (\lambda + \mu\, |t|)\, t\, e^{- \beta |t|}$ satisfies \eqref{2} and is odd. The assumptions \ref{A1} and \ref{A2} are easily verified if $\beta < 2 \sqrt{2 \alpha_0}$ and $\mu \ge 2 \alpha_0 \lambda/(2 \sqrt{2 \alpha_0} - \beta)$. Theorem \ref{Theorem 1} does not give a nontrivial solution of this problem for any $\mu > 0$ since $t h(t) \to 0$ as $|t| \to \infty$. However, Theorem \ref{Theorem 4} gives the following multiplicity result for large $\mu \ge 2 \alpha_0 \lambda/(2 \sqrt{2 \alpha_0} - \beta)$ since
\[
\int_0^t (\lambda + \mu\, |s|)\, s\, e^{\alpha_0 s^2 - \beta |s|}\, ds \ge \frac{\mu e^{- \beta^2/4 \alpha_0}}{3}\, |t|^3 \quad \forall t \in \R.
\]

\begin{theorem}
If $0 < \lambda < \lambda_1$, $0 < \beta < 2 \sqrt{2 \alpha_0}$, and
\[
\mu > \max \set{\frac{2 \alpha_0 \lambda}{2 \sqrt{2 \alpha_0} - \beta}, \frac{e^{\beta^2/4 \alpha_0}}{2} \left(\frac{\alpha_0}{3 \pi} \vol{\Omega}\right)^{1/2} \lambda_m^{3/2}},
\]
then problem \eqref{15} has $m$ distinct pairs of nontrivial solutions $\pm u_1,\dots,\pm u_m$. In particular, the number of solutions goes to infinity as $\mu \to \infty$.
\end{theorem}

The variational functionals associated with problems \eqref{1} and \eqref{3} satisfy the \PS{c} condition for all $c < 2 \pi/\alpha_0$ (see de Figueiredo et al.\! \cite[Proposition 2.1]{MR1386960}). The solutions of these problems given in Theorem \ref{Theorem 2} and Theorem \ref{Theorem 4} are critical points of the corresponding functional at critical levels below this threshold level for compactness. As we will see in the next section, the existence of these critical points is a consequence of an abstract critical point theorem recently proved in Perera \cite{Pe23} (see Theorem \ref{Theorem 3}).

\section{Proof of Theorem \ref{Theorem 2}}

Let us recall the definition of the $\Z_2$-cohomological index of Fadell and Rabinowitz \cite{MR0478189}. Let $W$ be a Banach space and let $\A$ denote the class of symmetric subsets of $W \setminus \set{0}$. For $A \in \A$, let $\overline{A} = A/\Z_2$ be the quotient space of $A$ with each $u$ and $-u$ identified, let $f : \overline{A} \to \RP^\infty$ be the classifying map of $\overline{A}$, and let $f^\ast : H^\ast(\RP^\infty) \to H^\ast(\overline{A})$ be the induced homomorphism of the Alexander-Spanier cohomology rings. The cohomological index of $A$ is defined by
\[
i(A) = \begin{cases}
\sup \set{m \ge 1 : f^\ast(\omega^{m-1}) \ne 0} & \text{if } A \ne \emptyset\\[5pt]
0 & \text{if } A = \emptyset,
\end{cases}
\]
where $\omega \in H^1(\RP^\infty)$ is the generator of the polynomial ring $H^\ast(\RP^\infty) = \Z_2[\omega]$. For example, the classifying map of the unit sphere $S^{m-1}$ in $\R^m,\, m \ge 1$ is the inclusion $\RP^{m-1} \incl \RP^\infty$, which induces isomorphisms on $H^q$ for $q \le m - 1$, so $i(S^{m-1}) = m$.

The cohomological index has the following so called piercing property, which is not shared by the Krasnoselskii’s genus. If $A, A_0, A_1 \in \A$ are closed and $\varphi : A \times [0,1] \to A_0 \cup A_1$ is a continuous map such that $\varphi(-u,t) = - \varphi(u,t)$ for all $(u,t) \in A \times [0,1]$, $\varphi(A \times [0,1])$ is closed, $\varphi(A \times \set{0}) \subset A_0$, and $\varphi(A \times \set{1}) \subset A_1$, then
\[
i(\varphi(A \times [0,1]) \cap A_0 \cap A_1) \ge i(A)
\]
(see \cite[Proposition (3.9)]{MR0478189}).

Now let $E \in C^1(W,\R)$ be an even functional, i.e., $E(-u) = E(u)$ for all $u \in E$. Assume that there exists $c^\ast > 0$ such that for all $c \in (0,c^\ast)$, $E$ satisfies the \PS{c} condition. Let $\A^\ast$ denote the class of symmetric subsets of $W$ and let $\Gamma$ denote the group of odd homeomorphisms of $W$ that are the identity outside the set $\set{u \in W : 0 < E(u) < c^\ast}$. For $r > 0$, the pseudo-index of $M \in \A^\ast$ related to $i$, $S_r = \set{u \in W : \norm{u} = r}$, and $\Gamma$ is defined by
\[
i^\ast(M) = \min_{\gamma \in \Gamma}\, i(\gamma(M) \cap S_r)
\]
(see Benci \cite{MR84c:58014}). Making essential use of the piercing property of the cohomological index, the following abstract critical point theorem was recently proved in Perera \cite{Pe23}.

\begin{theorem}[{\cite[Theorem 2.1]{Pe23}}] \label{Theorem 3}
Let $C$ be a compact symmetric subset of the unit sphere $S = \set{u \in W : \norm{u} = 1}$ with $i(C) = m \ge 1$. Assume that the origin is a strict local minimizer of $E$ and that there exists $R > 0$ such that
\begin{equation} \label{7}
\sup_{u \in A}\, E(u) \le 0, \qquad \sup_{u \in X}\, E(u) < c^\ast,
\end{equation}
where $A = \set{Ru : u \in C}$ and $X = \set{tu : u \in A,\, 0 \le t \le 1}$. Let $r \in (0,R)$ be so small that
\[
\inf_{u \in S_r}\, E(u) > 0,
\]
let $\A_j^\ast = \set{M \in \A^\ast : M \text{ is compact and } i^\ast(M) \ge j}$, and set
\[
c_j^\ast := \inf_{M \in \A_j^\ast}\, \sup_{u \in M}\, E(u), \quad j = 1,\dots,m.
\]
Then $0 < c_1^\ast \le \dotsb \le c_m^\ast < c^\ast$, each $c_j^\ast$ is a critical value of $E$, and $E$ has $m$ distinct pairs of associated critical points.
\end{theorem}

We now use this result to prove Theorem \ref{Theorem 2} and Theorem \ref{Theorem 4}.

\begin{proof}[Proof of Theorem \ref{Theorem 2}]
Since the variational functional
\[
E(u) = \frac{1}{2} \int_\Omega |\nabla u|^2\, dx - \mu \int_\Omega F(u)\, dx, \quad u \in H^1_0(\Omega)
\]
associated with problem \eqref{3} satisfies the \PS{c} condition for all $c < 2 \pi/\alpha_0$ by de Figueiredo et al.\! \cite[Proposition 2.1]{MR1386960}, we apply Theorem \ref{Theorem 3} with $c^\ast = 2 \pi/\alpha_0$. Clearly, \eqref{5} implies that the origin is a strict local minimizer of $E$.

Let $\lambda_1 < \lambda_2 \le \lambda_3 \le \cdots$ be the sequence of Dirichlet eigenvalues of $- \Delta$ in $\Omega$, repeated according to multiplicity, let $E_j$ denote the eigenspace of $\lambda_j$, and let
\[
N_m = \bigoplus_{j=1}^m E_j.
\]
By increasing $m$ if necessary, we may assume that $\lambda_m < \lambda_{m+1}$. Then the unit sphere $C$ in $N_m$ is a compact symmetric set of index $m$. Let $R > 0$ and let $A$ and $X$ be as in Theorem \ref{Theorem 3}. Integrating the inequality in \ref{A2} gives
\begin{equation} \label{8}
F(t) \ge F(t_0)\, e^{(|t| - t_0)/M} \quad \text{for } |t| \ge t_0.
\end{equation}
Since $F(t_0) > 0$ and $N_m$ is finite dimensional, this implies that $\sup E(A) \to - \infty$ as $R \to \infty$, so the first inequality in \eqref{7} holds if $R$ is sufficiently large.

For $u \in X \subset N_m$,
\begin{equation} \label{9}
E(u) \le \frac{\lambda_m}{2} \int_\Omega u^2\, dx - \mu \int_\Omega F(u)\, dx.
\end{equation}
By \eqref{8} and the first inequality in \ref{A1}, for any $\eps > 0$, there exists $c_\eps > 0$ such that
\begin{equation} \label{10}
F(t) \ge c_\eps\, |t|^3 \quad \text{for } |t| \ge \eps.
\end{equation}
We have
\[
\int_\Omega u^2\, dx = \int_{\set{|u| < \eps}} u^2\, dx + \int_{\set{|u| \ge \eps}} u^2\, dx \le \eps^2 \vol{\Omega} + \vol{\Omega}^{1/3} \left(\int_{\set{|u| \ge \eps}} |u|^3\, dx\right)^{2/3}
\]
by the H\"{o}lder inequality and
\[
\int_\Omega F(u)\, dx \ge \int_{\set{|u| \ge \eps}} F(u)\, dx \ge c_\eps \int_{\set{|u| \ge \eps}} |u|^3\, dx
\]
by \eqref{10}. So \eqref{9} gives
\[
E(u) \le \frac{\lambda_m}{2} \left(\eps^2 \vol{\Omega} + \vol{\Omega}^{1/3} \tau^2\right) - \mu c_\eps \tau^3 \quad \forall u \in X,
\]
where $\tau = \left(\int_{\set{|u| \ge \eps}} |u|^3\, dx\right)^{1/3}$. Maximizing the last expression over all $\tau \ge 0$ gives
\[
\sup_{u \in X}\, E(u) \le \frac{\lambda_m}{2} \left(\eps^2 + \frac{\lambda_m^2}{27 \mu^2 c_\eps^2}\right) \vol{\Omega},
\]
so the second inequality in \eqref{7} holds if $\eps$ is sufficiently small and $\mu$ is sufficiently large. Theorem \ref{Theorem 3} now gives $m$ distinct pairs of nontrivial critical points of $E$.
\end{proof}

\begin{proof}[Proof of Theorem \ref{Theorem 4}]
The proof is similar to that of Theorem \ref{Theorem 2}, so we will be sketchy. The variational functional
\[
E(u) = \frac{1}{2} \int_\Omega |\nabla u|^2\, dx - \int_\Omega F(u)\, dx, \quad u \in H^1_0(\Omega)
\]
associated with problem \eqref{1} satisfies the \PS{c} condition for all $c < 2 \pi/\alpha_0$ (see de Figueiredo et al.\! \cite[Proposition 2.1]{MR1386960}), so we apply Theorem \ref{Theorem 3} with $c^\ast = 2 \pi/\alpha_0$. Let $N_m$ and $C$ be as in the proof of Theorem \ref{Theorem 2}, let $R > 0$, and let $A$ and $X$ be as in Theorem \ref{Theorem 3}. As before, the first inequality in \eqref{7} holds if $R$ is sufficiently large. For $u \in X \subset N_m$, by \eqref{12} and the H\"{o}lder inequality,
\[
E(u) \le \frac{\lambda_m}{2} \int_\Omega u^2\, dx - \mu \int_\Omega |u|^s\, dx \le \frac{\lambda_m}{2} \vol{\Omega}^{1 - 2/s} \tau^2 - \mu \tau^s,
\]
where $\tau = \left(\int_\Omega |u|^s\, dx\right)^{1/s}$. Maximizing the last expression over all $\tau \ge 0$ gives
\[
\sup_{u \in X}\, E(u) \le \left(\frac{1}{2} - \frac{1}{s}\right) \bigg(\frac{\lambda_m^{s/2}}{\mu s}\bigg)^{2/(s - 2)} \vol{\Omega},
\]
which is less than $2 \pi/\alpha_0$ if \eqref{13} holds.
\end{proof}

\def\cprime{$''$}

\end{document}